\documentclass{IEEEtran}%
\usepackage{amssymb}
\usepackage{amsthm}
\usepackage{amsfonts}
\usepackage{graphicx}
\usepackage{amsmath}
\usepackage{geometry}
\usepackage{color}
\usepackage{fancyhdr}
\newtheorem{theorem}{Theorem}

\newtheorem{lemma}{Lemma}
\newtheorem*{notation*}{Notation}

\newtheorem*{proposition*}{Proposition}

\theoremstyle{definition}
\newtheorem{assumption}{Assumption}
\newtheorem{definition}{Definition}

\newcommand{\bqed}{\nopagebreak\mbox{}\hfill$\blacksquare$\smallskip}

\def\e{\epsilon}

\def \foral {\textrm{for all }}
\def \pr  {\mathbf{P}}
\def \R  {\mathbb{R}}
\def \E  {\mathbf{E}}

\begin{document}


\title{\vspace{0.2in}{\LARGE \textbf{The Reliability Value of Storage in a Volatile Environment}}}
\author{Ali ParandehGheibi, Mardavij Roozbehani, Asuman Ozdaglar, and Munther A Dahleh
\thanks{This work was supported by the National Science Foundation.}
\thanks{The authors are with the Laboratory for Information and Decision Systems, Department of Electrical Engineering and Computer Science, Massachusetts Institute of Technology, Cambridge, MA. Emails: \{parandeh, mardavij, asuman, dahleh\}@mit.edu.}%
}

\maketitle
\thispagestyle{empty}\pagestyle{empty}

\begin{abstract}
This paper examines the value of storage in securing reliability of a system with uncertain supply and demand, and supply friction. The storage is frictionless as a supply source, but once used, it cannot be filled up instantaneously. The focus application is a power supply network in which the base supply and demand are assumed to match perfectly, while deviations from the base are modeled as random shocks with stochastic arrivals. Due to friction, the random surge shocks cannot be tracked by the main supply sources. Storage, when available, can be used to compensate, fully or partially, for the surge in demand or loss of supply. The problem of optimal utilization of storage with the objective of maximizing system reliability is formulated as minimization of the expected discounted cost of blackouts over an infinite horizon. It is shown that when the stage cost is linear in the size of the blackout, the optimal policy is myopic in the sense that all shocks are compensated by storage up to the available level of storage. However, when the stage cost is strictly convex, it may be optimal to curtail some of the demand and allow a small current blackout in the interest of maintaining a higher level of reserve to avoid a large blackout in the future. The value of storage capacity in improving system's reliability, as well as the effects of the associated optimal policies under different stage costs on the probability distribution of blackouts are examined.
\end{abstract}

\begin{IEEEkeywords}
Storage, Ramp Constraints, Reliability, Probability of Large Blackouts
\end{IEEEkeywords}

\vspace{-0.05in}

\section{Introduction}
Supply and demand in electric power networks are subject to exogenous, impulsive, and unpredictable shocks due to generator outages, failure of transmission equipments or unexpected changes in weather conditions. On the other hand, environmental causes along with price pressure have led to a global trend in large-scale integration of renewable resources with stochastic output. This is likely to increase the magnitude and frequency of impulsive shocks to the supply side of the network. We ask, what is the value of storage in mitigating volatility of supply and demand, and what are the fundamental limits that cannot be overcome by storage due to physical ramp constraints, and finally, what are the impacts of different control policies on system reliability, for instance, on the expected cost or the probability of large blackouts?

In this paper our focus is on the reliability value of storage, defined as the maximal improvement in system reliability as a function of storage capacity. Two metrics for quantifying reliability in a system are considered: The first is the expected long-term discounted cost of blackouts (cost of blackouts (COB) metric), and the second is the probability of loss of load by a certain amount or less.

We model the system as a supply-demand model that is subject to random arrivals of energy deficit shocks, and a storage of limited capacity, with a ramp constraint on charging, but no constraint on discharging. The storage may be used to partially or completely mask the shocks to avoid blackouts. We formulate the problem of optimal storage management as the problem of minimization of the COB metric, and provide several characterizations of the optimal cost function. By ignoring other factors such as the environment, cost of energy or storage, we characterize the value of storage purely from a reliability perspective, and examine the effects of physical constraints on system reliability.  Moreover, for a general convex stage cost function, we present various structural properties of the optimal policy.

In particular, we prove that for a linear stage cost, a \emph{myopic} policy which compensates for all shocks regardless of their size by draining from storage as much as possible, is optimal. However, for nonlinear stage costs where the penalty for larger blackouts is significantly higher, the myopic policy is not optimal. Intuitively, the optimal policy is inclined to mitigate large blackouts at the cost of allowing more frequent small blackouts. Our numerical results confirm this intuition. We further investigate the value of additional storage under different control policies, and for different ranges of system parameters. Our results suggest that if the ratio of the average rate of deficit shocks to ramp constraints is sufficiently large, there is a critical level of storage capacity above which, the value of having additional capacity quickly diminishes. When this ratio is significantly large, there seems to be another critical level for storage size below which, storage capacity provides very little value. Finally, we investigate the effect of storage size and volatility of the demand/supply process on the probability of large blackouts under various policies. We observe that for all control policies, there appears to be a critical level of storage size, above which the probability of suffering large blackouts diminishes quickly.

Recent works have examined the effects of ramp constraints on the economic value of storage \cite{MVichCDC2011_2}. Herein, our focus is on reliability. Prior research on using queueing models for characterization of system reliability, particularly in power systems, has been reported in
\cite{chechomey06a} and \cite{chomey09a}. Similar models and concepts exist in the queueing theory literature \cite{Gross}, \cite{MeynBook07}, perhaps with different application contexts. Despite similarities, our model is different than those of \cite{chechomey06a}, \cite{chomey09a} in many ways. We assume that the storage capacity (reserve in their model) is fixed and find the optimal policy for withdrawing from storage (consuming from reserve), as opposed to always draining the reserve and optimizing the capacity. Another difference is that our model of uncertainty is a compound poisson process instead of the brownian motion used in \cite{chechomey06a}, \cite{chomey09a}. We show that the myopic policy of always draining storage to mask every energy deficit shock is not optimal for strictly convex costs, and investigate the effects of nonlinear stage costs (strictly convex cost of blackouts) on the optimal policy and the statistics of blackouts.

The organization of this paper is as follows. Section \ref{sec:modeling} presents the elements of the model and the problem formulation. Section \ref{sec:main} includes the main analytical results. Section \ref{sec:numsim} presents the numerical simulations and discussions. Finally, Section \ref{sec:conclusions} includes the concluding remarks.

\begin{notation*}
Throughout the paper, $\mathbb{I}_A$ denotes the indicator function of a set $A$. The operator $[x]^+=\max\{0,x\}$ is the projection operator onto the nonnegative orthant.
\end{notation*}

\thispagestyle{empty}\pagestyle{empty}

\section{The Model} \label{sec:modeling}

We examine an abstract model of system consisting of a single consumer, a single fully controllable supplier, a supplier with stochastic output (e.g., wind), and a storage system with finite capacity (Figure \ref{fig:powersys}). These agents each represent an aggregate of several small consumers and producers. The details of the model are outlined below.
\begin{figure}[h]
\begin{center}
\includegraphics[scale=0.25]{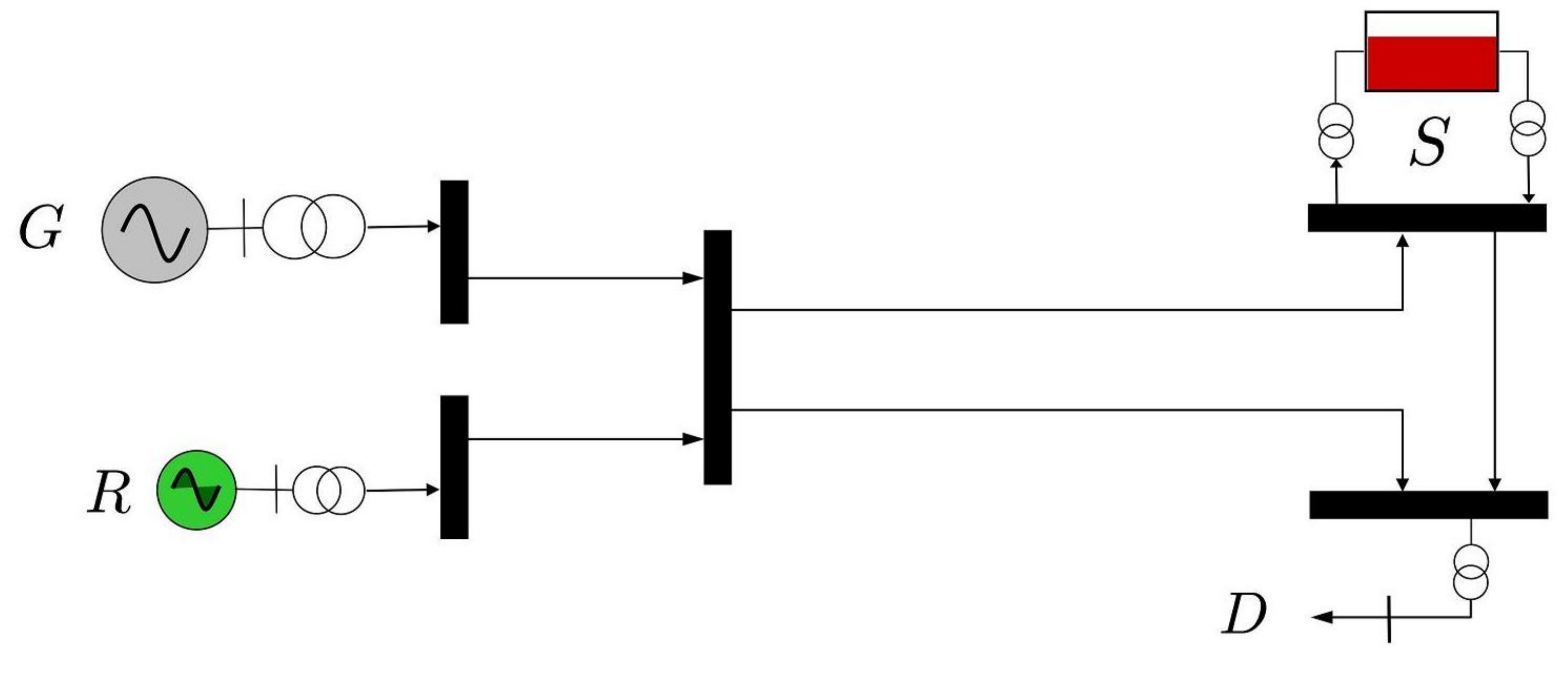}
\vspace{-0.1in}
\caption{Layout of the physical layer of a power supply network with conventional and renewable generation, storage, and demand.}
\label{fig:powersys}
\end{center}
\end{figure}
\vspace{-0.3in}
\subsection{Supply}
\subsubsection{Controllable Supply}
The controllable supply process is denoted by ${\bf G}=\{G_t:t\geq0\}$, where $G_t$ is the power output at time $t\geq0$. It is assumed that the supplier's production is subject to an upward ramp constraint, in the sense that its output cannot increase instantaneously,
\[
\frac{G_t-G_{t'}}{t-t'}\leq\zeta, \qquad \forall t: 0\leq t < t'.
\]
We do not assume a downward ramp constraint or a maximum capacity constraint on $G_t$. Thus, production can shut down instantaneously, and can meet any large demand sufficiently far in the future.
\subsubsection{Renewable Supply}
The renewable supply process is denoted by ${\bf R}=\{R_t:t\geq0\}.$ It is assumed that ${\bf R}$ can be modeled as a process with two components: ${\bf R}={\bf \overline R}+\Delta{\bf R}$, where ${\bf \overline R}=\{\overline R_t:t\geq0\}$ is a deterministic process representing the predicted renewable supply, and $\Delta{\bf R}=\{\Delta R_t:t\geq0\}$ is the residual supply assumed to be a random arrival process. Thus, at any given time $t\geq 0,$ the total forecast supply from the renewable and controllable generators is given by  ${ G_t} + {\overline R_t}.$


\subsection{Demand}
The demand process is denoted by ${\bf D}=\{D_t:t\geq0\},$ where $D_t$ is the total power demand at time $t$, assumed to be exogenous and inelastic. Similar to the renewable supply, ${\bf D}$ has two components: ${\bf D}={\bf \overline D}+\Delta{\bf D}$, where ${\bf \overline D}=\{\overline D_t:t\geq0\}$ is the predicted demand process (deterministic), and $\Delta{\bf D}=\{\Delta D_t:t\geq0\}$ is the residual demand, again, assumed to be a random arrival process.

\begin{definition}
The \emph{power imbalance} is defined as the residual demand minus the residual supply.
\begin{equation}\label{eq:NIB}
P_t= \Delta D_t-\Delta R_t
\end{equation}
The \emph{normalized energy imbalance} is defined as:
\begin{equation}\label{eq:ENIB}
W_t =\frac{P_t^2}{2\zeta}
\end{equation}
\end{definition}
\vspace{-0.4in}
\subsection{Storage}
The storage process is denoted by ${\bf s}=\{s_t\in[0,\overline{s}]:t\geq0\}$, where $s_t$ is the amount of stored energy at time $t,$ and $\overline{s}<\infty$ is the storage capacity. The storage technology is subject to an upward ramp constraint:
\[
\frac{s_t-s_{t'}}{t-t'}\leq r, \qquad \forall t: 0\leq t < t'.
\]
Thus, storage cannot be filled up instantaneously, though, it can be drained (to supply power) instantaneously. Let ${\bf U}=\{U_t:t\geq0\},$ be the power withdrawal process from storage. The dynamics of storage is then given by:
\begin{align}\label{eq:StateEv0}
s_t&=s_0+\int_0^t\mathbb{I}_{\{s_\tau<\overline{s}\}}rd\tau-\int_0^t U_\tau d{\tau}
\end{align}
It is desired to design a causal controller ${K}$ such that the control law $U_t=K(s_t,G_t+R_t-D_t)$ maximizes the system reliability objectives.

\begin{figure}[h]
\begin{center}
\includegraphics[scale=0.25]{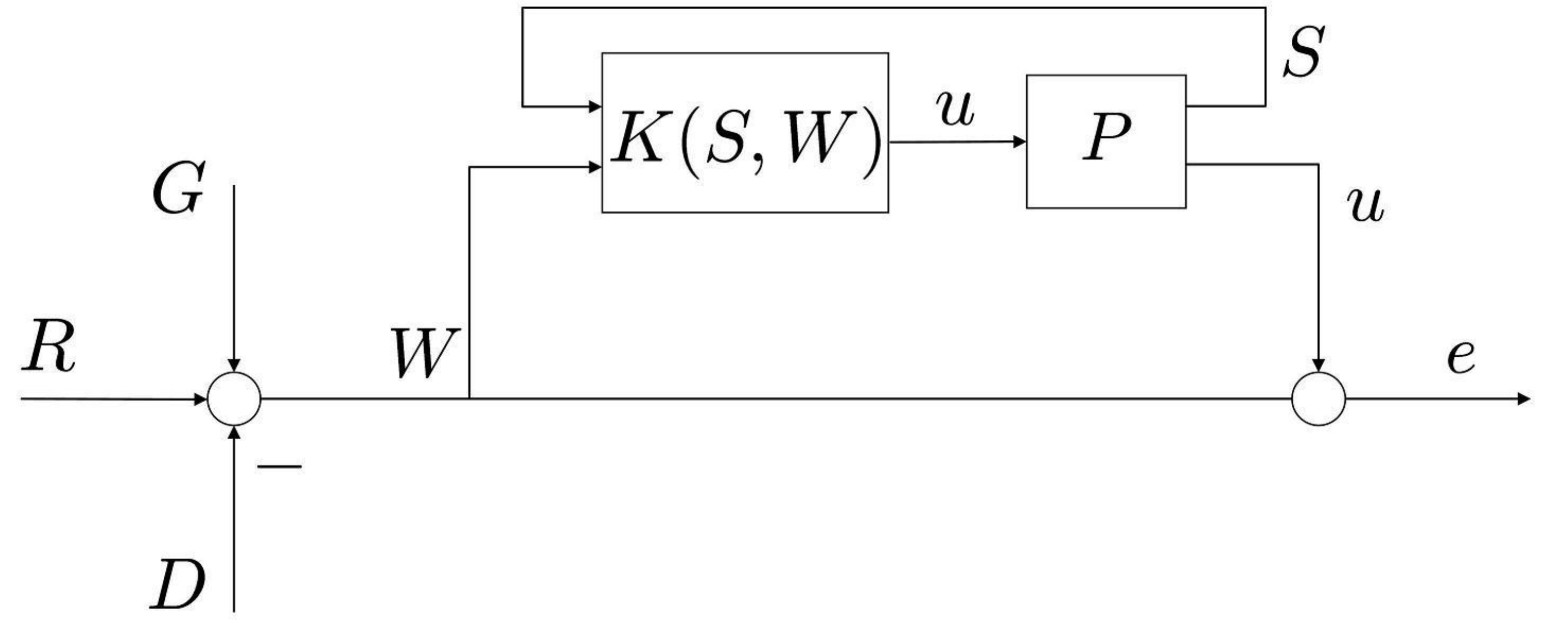}
\vspace{-0.05in}
\caption{The control layer of the power supply network in Figure \ref{fig:powersys}.}
\label{fig:controlsys}
\end{center}
\end{figure}
\vspace{-0.3in}
\subsection{Reliability Metric}
We refer to the event of not meeting the demand as a \emph{blackout}. The cost of blackouts (COB) metric is defined as the expected long-term discounted cost of blackouts:
\begin{eqnarray}\label{def:metric}
    C^\textrm{bo} \hspace{-8pt} &=& \hspace{-8pt} \E\bigg[\int_{0}^{\infty} e^{-\theta \tau} h\left([P_\tau]^+\right) d{\tau} \bigg ]
\end{eqnarray}
where $P(\cdot)$ is the power imbalance process, $h:\mathbb{R}_+\mapsto\mathbb{R}_+$ is an increasing function, and $\theta>0$ is the discount rate.
\subsection{Problem Formulation}
In this section we present the problem formulation. Before we proceed, we pose the following assumptions.
\begin{assumption}\label{Assum1}
The normalized energy imbalance process \eqref{eq:ENIB} is the jump process in a compound poisson process with arrival rate $Q$ and jump size distribution $f_W,$ where the support of $f_W$ lies within a bounded interval $[0,B].$ The maximum jump size is thus upperbounded by $B.$
\end{assumption}

\begin{assumption}\label{Assum2}
The forecast supply is equal to the forecast demand. That is:
\[
\overline D_t=G_t+\overline R_t, \qquad \forall t\geq 0
\]
\end{assumption}

Under Assumption \ref{Assum2}, the energy from storage will be used only to compensate for the power imbalance, since in the absence of an \emph{energy} shock, supply is equal to demand, and storage provides no additional utility.
Under Assumptions \ref{Assum1} and \ref{Assum2}, the dynamics of the storage process can be written as:
\begin{align}\label{eq:StateEv1}
s_t&=s_0+\int_0^t\mathbb{I}_{\{s_\tau<\overline{s}\}}rd\tau-\int_0^t \mu\left(s_{\tau^-},W_{\tau}\right)dN_{\tau}
\end{align}
where $N_t$ is a Poisson process of rate $Q$, and $W_t$ is the jump size (energy imbalance) process, drawn independently and identically from a distribution $f_W.$ Further, $\mu$ denotes a \emph{control policy}. We focus on stationary Markov policies since the energy imbalance modeled as a compound Poisson process is stationary and memoryless.  We denote the set of all such feasible policies by $\Pi$.

We are now ready to state the problem formulation.
Let $C_\mu(s)$ denote the expected long-term discounted cost of blackouts starting from an initial state $s$ and under control policy $\mu$,
\begin{eqnarray}\label{discounted_cost}
    C_\mu(s) \hspace{-8pt} &=& \hspace{-8pt} \E\bigg[\sum_{k= 1}^{\infty} e^{-\theta t_k} g\left(W_k - \mu(s_{t^-_k}, W_k)\right) \Big| s_0 = s \bigg ],\hspace{0.2in}
\end{eqnarray}
where $t_k$ is the $k$-th Poission arrival time, and $W_k = W_{t_k}$ is size of the $k$-th jump. Moreover, $g: [0,B] \rightarrow \R$ is the stage cost as a function of energy imbalance (blackout size). In this work, we assume the following assumptions hold.

\begin{assumption}\label{assump:stage_cost}
The stage cost function $g(\cdot)$ is bounded, strictly increasing and continuously differentiable. Moreover, $\E_W[g(W)] > 0$, and $g(0) = 0$.
\end{assumption}

The system reliability problem can now be formulated as an infinite horizon stochastic optimal control problem
\begin{equation}\label{eq:main}
C_\mu(s) \rightarrow \min_{\mu\in\Pi}
\end{equation}
where the optimization problem \eqref{eq:main} is subject to the state dynamics \eqref{eq:StateEv1}.
A policy $\mu^* \in \Pi$ is defined to be optimal if
\[
\mu^*\in\arg\min_{\mu\in\Pi} C_\mu(s).
\]
The associated \emph{value function} or optimal cost function is denoted by $C(s)$, where
\begin{equation}\label{value_function}
C(s) = \min_{\mu \in \Pi} C_\mu(s), \quad 0\leq s \leq \bar s.
\end{equation}
\section{Main Results}\label{sec:main}
\subsection{Characterizations of the Value Function}\label{sec:optimal_cost}
We first provide several characterizations for the value function defined in (\ref{value_function}) and establish specific properties that are useful in characterization of the optimal policy.

Let $J_\mu(s, w)$ be the expected long-term discounted cost under policy $\mu$ conditioned on the first jump arriving at time $t_1 = 0$, and being of size $w$.
Here, $s$ is the state of the system before executing the action dictated by the policy. By the memoryless property of the Poisson process, we have
\begin{eqnarray}\label{J_mu}
&&  \hspace{-27pt} J_\mu(s, w) = g(w - \mu(s,w)) \nonumber \\
&& \hspace{-25pt}  + \E\Big[\sum_{k= 1}^{\infty} e^{-\theta t_k} g(W_k \hspace{-0.01in}-\hspace{-0.01in} \mu(s_{t^-_k}, W_k)) \Big| s_0\hspace{-0.01in} = \hspace{-0.01in}s-\mu(s,w) \Big ]\hspace{0.15in}
\end{eqnarray}

We may relate $J_\mu(s,W)$ to the total expected cost $C_\mu(s)$ defined in (\ref{discounted_cost}) as follows:
\begin{equation}\label{J_C_relation}
    C_\mu(s) = \E\left[e^{-\theta t_0} J_\mu(\min\{s + r t_0, \bar s\}, W)\right],
\end{equation}
where $t_0$ is an exponential random variable with mean $1/Q$, and is independent of $W$, drawn from distribution $f_W$.

From (\ref{J_C_relation}), it is clear that from the minimization of  $J_\mu$ across all admissible policies $\Pi$, we may obtain the optimal solution to the original problem in (\ref{value_function}). The discrete-time formulation of  $J_\mu$ given by (\ref{J_mu}), facilitates deriving the Bellman equation as the necessary and sufficient optimality condition, as well as development of efficient numerical methods. We summarize these results in the following theorem.

\begin{theorem}\label{thm:bellman}
Given an admissible control policy $\mu\in\Pi$, let $J_\mu:[0,\overline{s}]\times[0,B]\mapsto\mathbb{R}$ be the function defined as in (\ref{J_mu}). A function $J:[0,\overline{s}]\times[0,B]\mapsto\mathbb{R}$ satisfies
\[
J(s,w)=J^*(s,w) \overset{\emph{def}}= \min_{\mu \in \Pi} J_\mu(s,w),\qquad\forall (s,w),
\]
if and only if it satisfies the following fixed-point equation:
\begin{eqnarray}\label{bellman_eq}
    J(s,w) \!\!\!\!\! &=& \!\!\!\!\! (T J)(s,w) \overset{\emph{def}}= \min_{u\in[0,\min\{s,w\}]} \bigg\{g(w - u) \nonumber \\
    && \hspace{-10pt}+ \ \E\Big[e^{-\theta t_0} J\big(\min\{s - u + r t_0, \bar s\}, W\big)  \Big]\bigg\},\hspace{0.25in}
\end{eqnarray}
Moreover, a stationary policy $\mu^*(s,w)$ is optimal if and only if $u=\mu^*(s,w)$ achieves the minimum in (\ref{bellman_eq}) for $J=J^*$. Finally, the value iteration algorithm
\begin{equation}\label{value_iteration}
    J_{k+1} = T J_{k},
\end{equation}
converges to  $J^*$ for any initial condition $J_0$.
\end{theorem}
\begin{proof}
The result follows from establishing the contraction property of $T$, which is standard for discounted problems with bounded stage cost. See \cite{DP_book} for more details.
\end{proof}

An alternative approach to characterization of the optimal cost function is based on continuous-time analysis of problem (\ref{value_function}), which leads to Hamilton-Jacobi-Bellman (HJB) equation. In the following theorem we present some basic properties of the optimal cost function as well as the HJB equation.

\begin{theorem}\label{thm:HJB}
Let $C(s)$ be the optimal cost function defined in (\ref{value_function}). The following statements hold:
\begin{description}
\item[\hspace{0.04in}  (i)] $C(s)$ is strictly decreasing in $s$.\vspace{0.08in}
  \item[\hspace{0.011in}  (ii)]  If the stage cost $g(\cdot)$ is convex, the optimal cost function $C(s)$ is also convex in $s$.\vspace{0.08in}
  \item[\hspace{-0.015in}  (iii)]
       If $C$ is continuously differentiable, then for all $s\in[0,\bar s],$ it satisfies the following HJB equation
  \end{description}
       \begin{eqnarray}\label{HJB_eq}
           \hspace{-10pt} \frac{d C}{d s} \hspace{-5pt} &=& \hspace{-5pt} \frac{Q+\theta}{r}  C(s) \nonumber \\
           &- & \hspace{-5pt} \frac{Q}{r} \E\Big[\min_{u\in[0, \min{\{s,W\}}]} g(W-u) + C(s-u) \Big],
       \end{eqnarray}
       with the boundary condition
       \begin{equation}\label{HJB_BC}
           \frac{d C}{d s}\Big|_{s = \bar s} = 0.
       \end{equation}
       Moreover, the optimal policy $\mu^*(s,w)$ achieves the optimal solution of the minimization problem in (\ref{HJB_eq}). Furthermore, for a given policy $\mu$, if the cost function $C_\mu(s)$ is differentiable, it satisfies the following delay differential equation
       \begin{eqnarray}\label{C_mu_DDE}
            \hspace{-0.2in}\frac{d C_\mu}{d s}\hspace{-0.05in} &=& \hspace{-0.05in}\frac{Q+\theta}{r}  C_\mu(s)  \nonumber\\
            \hspace{-0.05in}&-&\hspace{-0.05in}\frac{Q}{r} \E\Big[g(W-\mu(s,W)) + C_\mu(s-\mu(s,W)) \Big],
       \end{eqnarray}
       with the boundary condition given by (\ref{HJB_BC}).
\end{theorem}
\begin{proof}
See the Appendix.
\end{proof}

The result of Theorem \ref{thm:HJB} part (iii) requires continuous differentiability of the optimal cost function, which can be established under some mild conditions such as differentiability of the stage cost function $g$ and the probability density function $f_W(\cdot)$ of Poisson jumps (cf. Benveniste and Scheinkman \cite{Benveniste79}). Throughout this paper, we assume that $C(s)$ is in fact continuously differentiable and the results of Theorem \ref{thm:HJB} are applicable.

\subsection{Characterizations of the Optimal Policy}\label{sec:optimal_policy}
In this subsection, we derive some structural properties of the optimal policy using the optimal cost characterizations given in Theorems \ref{thm:bellman} and \ref{thm:HJB}. First, we show that the myopic policy of allocating reserve energy from storage to cover as much of every shock as possible is optimal for linear stage cost functions. Then, we partially characterize the structure of optimal policy for strictly convex stage cost functions.

\begin{theorem}\label{thm:linear_policy}
If the stage cost is linear, i.e., $g(x) = \beta x$ for some $\beta>0$, then the myopic policy
\begin{equation}\label{myopic_plicy}
    \mu^*(s,w) = \min\{s, w\},
\end{equation}
is optimal for problem (\ref{value_function}).
\end{theorem}
\begin{proof}
See the Appendix.
\end{proof}

Next, we focus on nonlinear but convex stage cost functions. In this case, the myopic policy defined in (\ref{myopic_plicy}) is no longer optimal. Intuitively, the myopic policy greedily consumes the reserve and thereby increases the chance of a large blackout. In the linear stage cost case, the penalty for a large blackout is equivalent to the total penalty of many small blackouts. This is contrary to the strictly convex case. Therefore, the optimal policy in this case tends to be more conservative in consuming the reserve. Nevertheless, the structure of the optimal policy shows some similarities with the myopic policy. In the following we present some characterizations of the structural properties of the optimal policy using the results from Section \ref{sec:optimal_cost}.

\begin{assumption}\label{assump:neg_drift}
The storage process has a positive drift in the sense that the rate of the compound Poisson process is less than the ramp constraint, i.e.,
\[{Q \E[W]} \leq r.\]
\end{assumption}

\begin{theorem}\label{thm:policy_monotonicity}
    Let $\mu^*(s,w)$ be the optimal policy associated with problem \eqref{value_function}. If Assumption \ref{assump:neg_drift} holds, then $\mu^*(s,w)$ is monotonically nondecreasing in both $s$ and $w$.
\end{theorem}
\begin{proof}
See the Appendix.
\end{proof}

\begin{theorem}\label{thm:collapse}
Let $\mu^*$ denote the optimal policy associated with problem (\ref{value_function}) with strictly convex stage cost $g(\cdot)$. There exist a unique kernel function $\phi: [-B,\bar s] \rightarrow \R$ such that
\begin{equation}\label{collapse_eq}
\hspace{-0.02in}   \mu^*(s,w)\hspace{-0.02in} =\hspace{-0.02in} \Big[w - \phi(s-w)\Big]^+, \hspace{-0.05in}\quad \forall (s,w) \in [0,\bar s]\times [0,B],
\end{equation}
where,
\vspace{-0.1in}
\begin{align}
&
\begin{array}
[c]{cccl}%
\phi\left(  p\right)   & = & \arg\min\limits_{x} & g\left(  x\right)
+C\left(  x+p\right)
\end{array}
\label{phi_def}\\
&
\begin{array}
[c]{cccl}
& \hspace{0.65in} & \emph{s.t.} & x\leq\min\left\{  B,\bar{s}-p\right\}
\vspace{0.07in}\\
& \hspace{0.65in} &  & x\geq\max\left\{  0,-p\right\}
\end{array}
\nonumber
\end{align}

\noindent Moreover, under Assumption \ref{assump:neg_drift}, we can represent the kernel function $\phi(p)$ as follows:
\begin{equation}\label{kernel_partition}
    \phi(p) = \left\{
                \begin{array}{ll}
                  -p, & \hbox{$-B \leq p \leq b_0$} \\[0.05in]
                  \phi^\circ(p), & \hbox{$\ \ b_0 \leq p \leq b_1$} \\[0.05in]
                  0, & \hbox{$\ \ b_1 \leq p \leq \bar s$,}
                \end{array}
              \right.
\end{equation}
where $\phi^\circ(p)$ is the unique solution of
\begin{equation}\label{interior_eq}
    g'(x)  + C'(x+p) = 0,
\end{equation}
and $b_0$ and $b_1$ are the break-points, where
\begin{multline}\label{b0}
    b_0 = -(g')^{(-1)}\Big(C'(0)\Big)\\
    \geq -(g')^{(-1)}\Big(\frac{Q}{r}\E[g(W)]\Big) \geq -B,
\end{multline}
\begin{equation}\label{b1}
    b_1 = -(C')^{(-1)}\Big(g'(0)\Big) \leq \bar s.
\end{equation}
\end{theorem}
\begin{proof}
See Appendix.
\end{proof}

Theorem \ref{thm:collapse} demonstrates a very special structure for the optimal policy. In fact, it shows that the two dimensional policy can be represented using a single dimensional kernel function. This result allows us to significantly reduce the computational complexity of numerical methods for computing the optimal policy. In addition, using Theorem \ref{thm:collapse}, we can provide a qualitative picture of the structure of the optimal policy.  Figures \ref{fig:kernel_concept} and \ref{fig:policy_concept} illustrate a conceptual plot of the kernel function, and the optimal policy, respectively.

\begin{figure}[htbp]
\centering
  \includegraphics[width=2.5in]{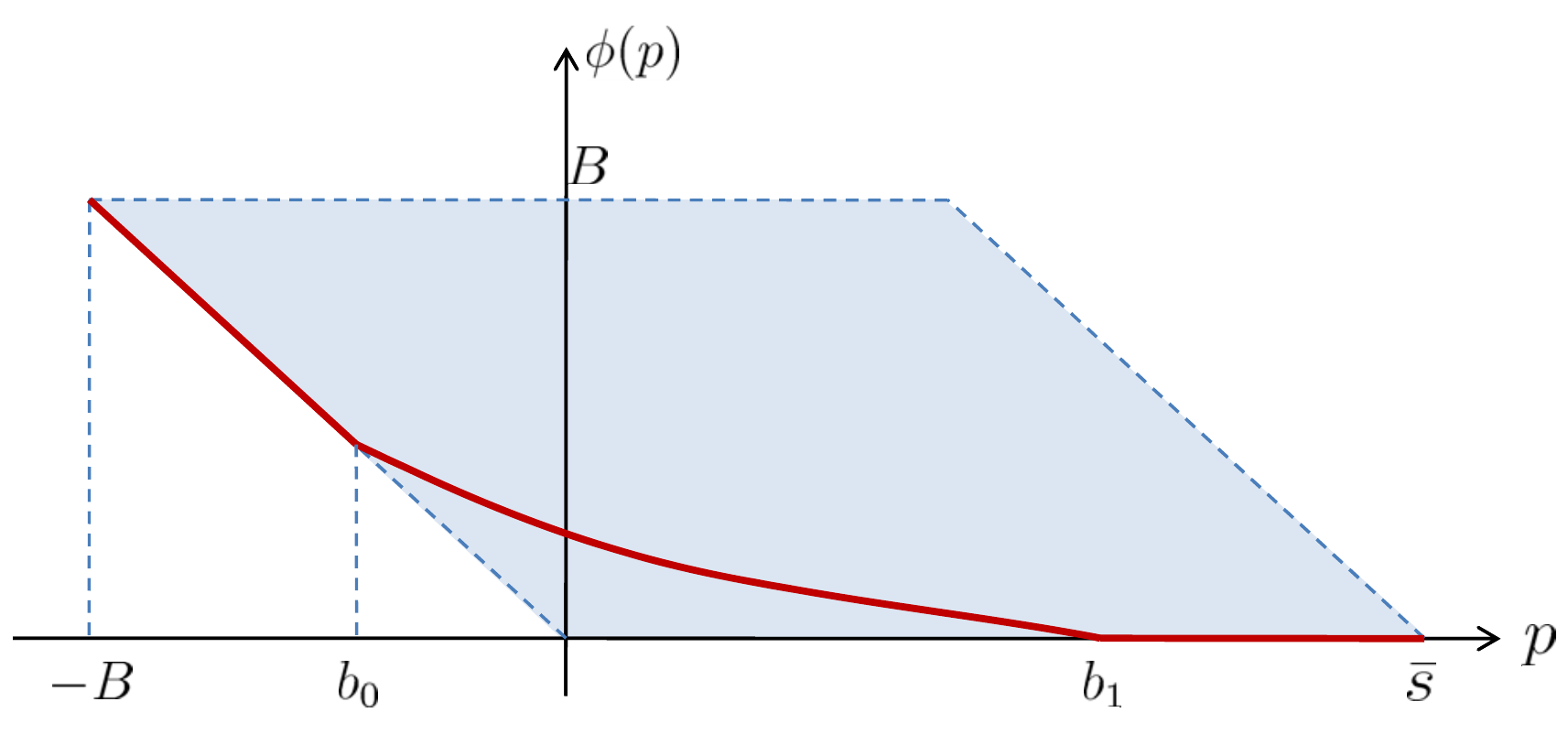}
  \caption{Structure of the kernel function $\phi(p)$ defined in (\ref{phi_def}).}\label{fig:kernel_concept}
\end{figure}

\begin{figure}[htbp]
\centering
  \includegraphics[width=2.5in]{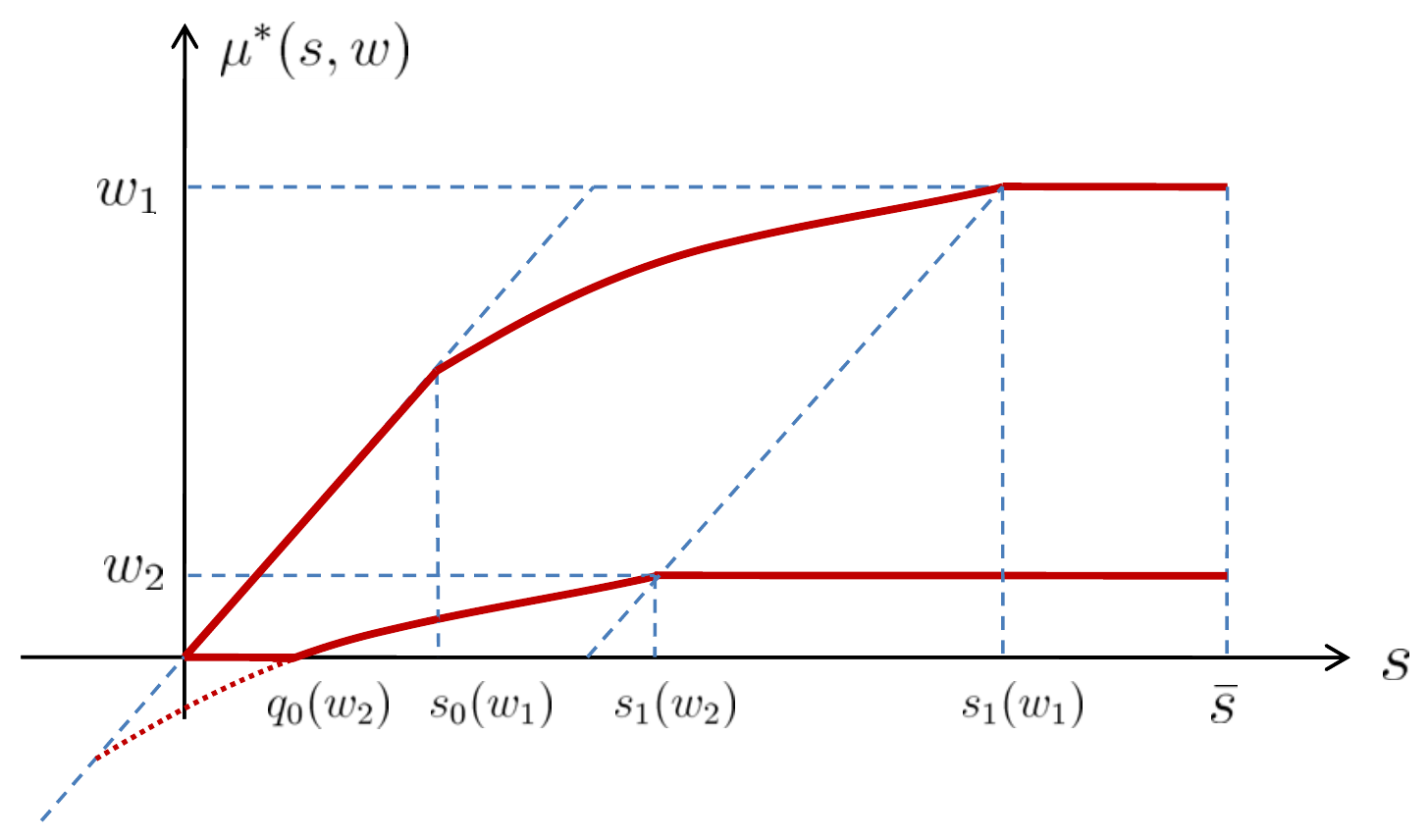}
  \caption{Structure of the optimal policy $\mu^*(s,w)$ for a convex stage cost, for $w = w_1, w_2$.}\label{fig:policy_concept}
  \vspace{-0.2in}
\end{figure}

In particular, we can summarize the characterization of the optimal policy as follows. If $w \geq -b_0$, we have
\begin{equation}\label{policy_partition1}
    \mu^*(s,w) = \left\{
                \begin{array}{ll}
                  s, & \hbox{$\qquad 0 \leq s \leq s_0(w)$} \\
                  w - \phi^\circ(s-w), & \hbox{$s_0(w) \leq s \leq s_1(w)$} \\
                  w, & \hbox{$s_1(w) \leq s \leq \bar s$,}
                \end{array}
              \right.
\end{equation}
where $s_i(w) = w+b_i$ for $i=0,1$. In the case where  $w \leq -b_0$, we have
\begin{equation}\label{policy_partition2}
    \mu^*(s,w) = \left\{
                \begin{array}{ll}
                  0, & \hbox{$\qquad 0 \leq s \leq q_0(w)$} \\
                  w - \phi^\circ(s-w), & \hbox{$q_0(w) \leq s \leq s_1(w)$} \\
                  w, & \hbox{$s_1(w) \leq s \leq \bar s$,}
                \end{array}
              \right.
\end{equation}
where $q_0(w)$ is the unique solution of $\phi^\circ(s-w) = w$.

\section{Numerical Simulations}\label{sec:numsim}
In this part, we present numerical characterizations of the optimal cost function and optimal policy in different scenarios. Moreover, we study the effect of storage size and volatility on system performance, for various control policies.

We use the value iteration algorithm (\ref{value_iteration}) to compute the optimal policy and cost function for nonlinear stage costs. Figures \ref{fig:policy_numerical} and \ref{fig:cost_numerical} illustrate the optimal policy and cost function in a scenario with uniformly distributed random jumps, quadratic stage cost, and the following parameters: $\theta = 0.1, r = 1, Q = 0.8, \bar s = 2$. Observe that the optimal policy complies with the conceptual Figure \ref{fig:policy_concept}.


\begin{figure}[h]
\centering
  \includegraphics[width=2.5in]{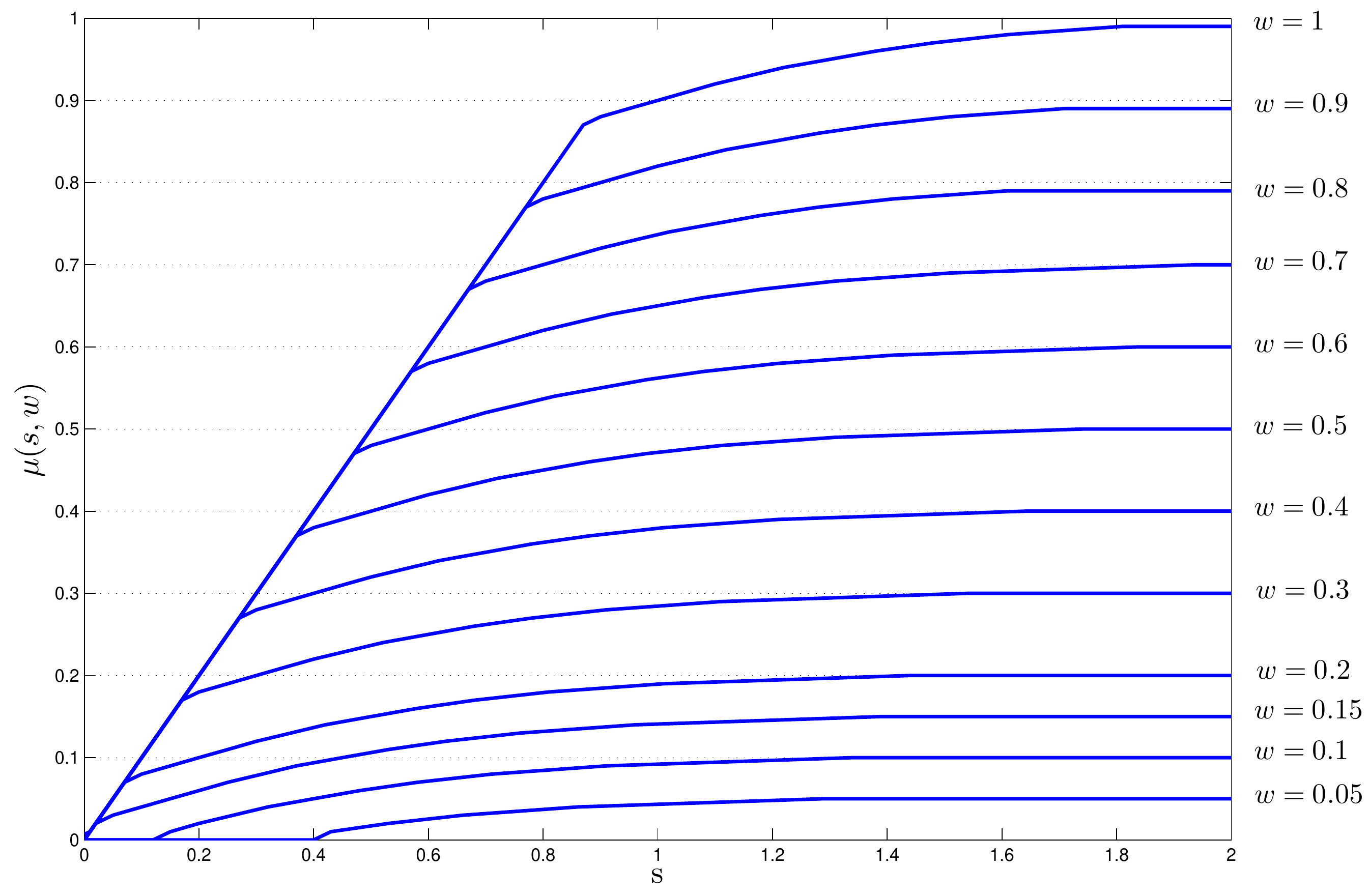}
  \caption{Optimal policy computed by value iteration algorithm (\ref{value_iteration}) for quadratic stage cost and uniform shock distribution.}\label{fig:policy_numerical}
\end{figure}

\begin{figure}[h]
\centering
  \includegraphics[width=2.2in]{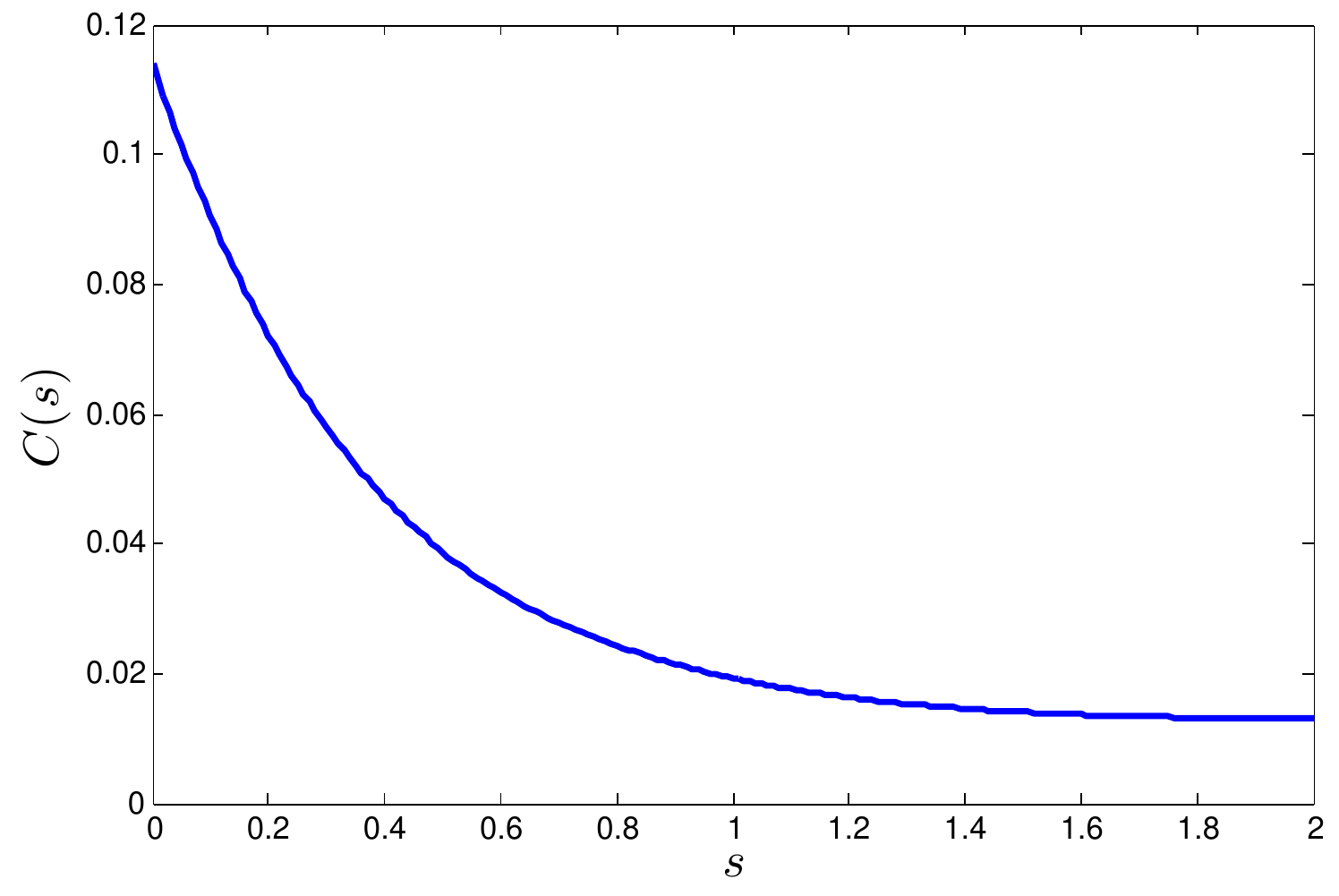}
  \caption{Optimal cost function computed by value iteration algorithm (\ref{value_iteration}) for quadratic stage cost and uniform shock distribution.}\label{fig:cost_numerical}
\end{figure}

\begin{figure}[h]
\centering
  \includegraphics[width=2.5in]{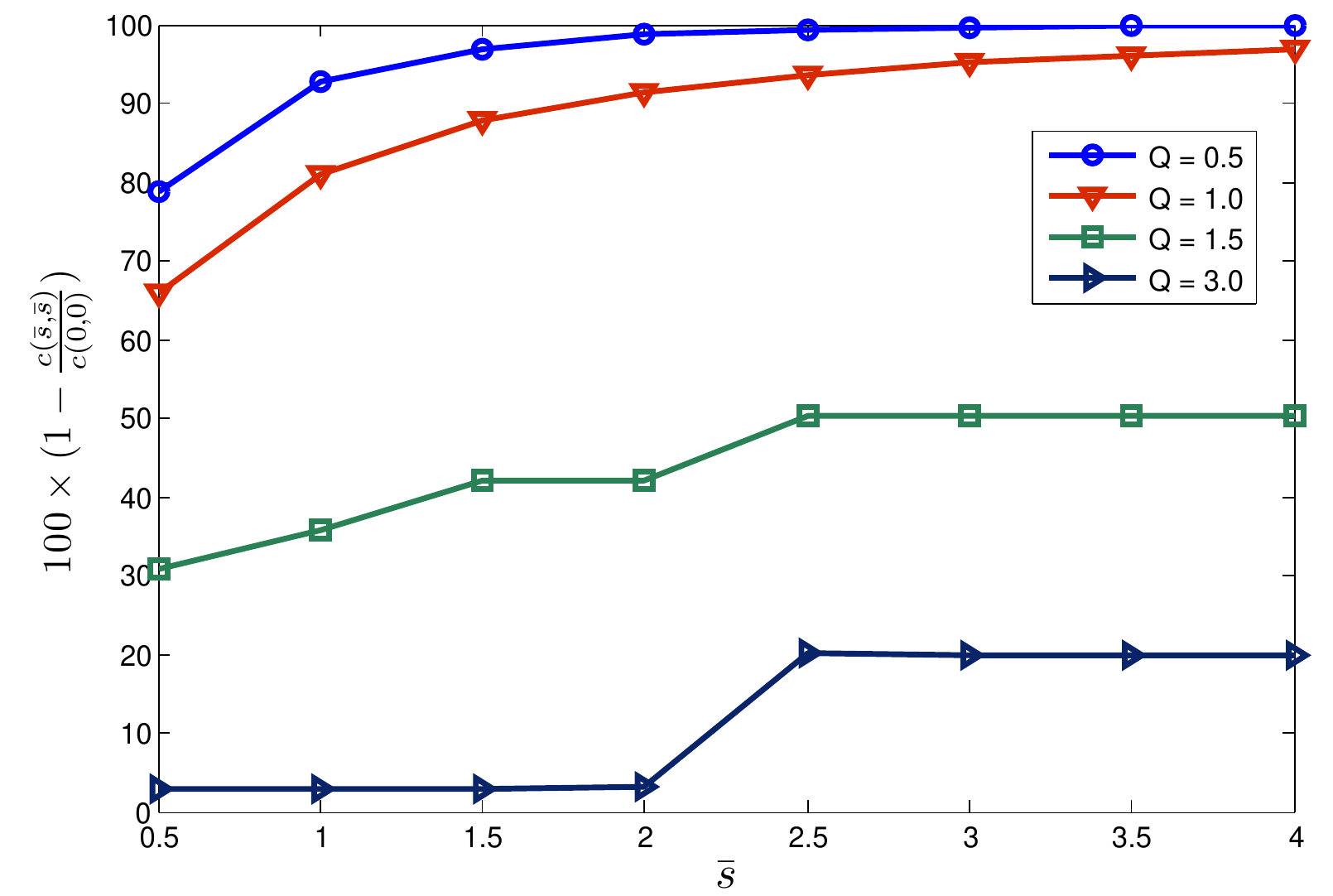}
  \caption{Value of energy storage as a function of the storage capacity for different Poisson arrival rates. $c(s; \bar s)$ denotes the optimal cost function (\ref{value_function}) when the storage capacity is given by $\bar s$.}\label{fig:s_bar_effect}
  \vspace{-.1in}
\end{figure}

Figure \ref{fig:s_bar_effect} shows the value of storage, defined as the normalized improvement of energy storage in expected cost, for different Poisson arrival rates. In this case $\theta = 0.01, g(x) = x^3, r = 1, W = 1$. Note that the storage process has a  negative drift if and only if $Q > 1$. Observe that in the positive or zero drift cases, even a small value of storage yields a significant effect in reducing the blackout cost. However, in the negative drift case, the value of storage is significantly lower. Observe that for the negative drift case, there is a critical storage size that yields a sharp improvement in the value of storage.

\vspace{-.1in}
\subsection{Blackout Statistics}\label{sec:BO_stat}
 We discussed in Section \ref{sec:optimal_policy} that the myopic policy given by (\ref{myopic_plicy}) is not necessary optimal for nonlinear stage cost functions. In this part, we study the effect of different optimal policies, in the sense of (\ref{eq:main}), for different stage costs on the distribution of large blackouts. Figure \ref{fig:BO_dist} shows the blackout distribution in a scenario with deterministic jumps of size one, for both myopic policy and the optimal policy for a cubic cost function. Note that, the stage cost for the non-myopic policy assigns a significantly higher weight to larger blackouts. Therefore, as we can see in Figure \ref{fig:BO_dist}, the non-myopic policy results in less frequent large blackouts at the price of more frequent small blackouts.

\begin{figure}[h]
\centering
   \vspace{-.1in}
  \includegraphics[width=2.75in]{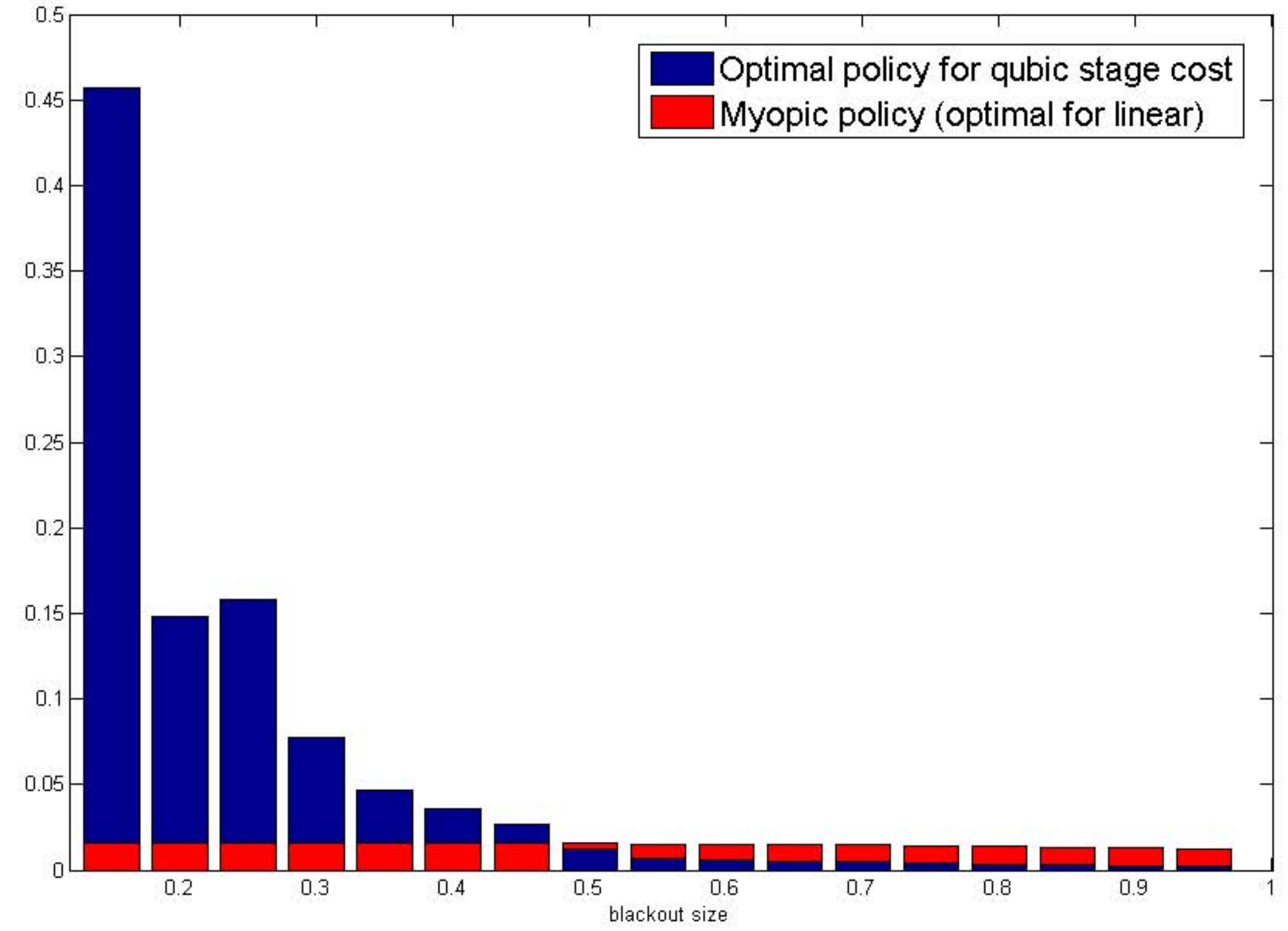}
  \caption{Blackout distribution comparison of myopic and non-myopic policies (deterministic jumps with rate $Q = 0.8$).}\label{fig:BO_dist}
    \vspace{-.2in}
\end{figure}

Next, we study the effect of storage size on probability of large blackouts. Figure \ref{fig:BO_tail_s_bar} plots this metric for different policies that are all optimal for different stage cost functions. Similarly to Figure \ref{fig:s_bar_effect}, we observe a sharp improvement of the reliability metric at a critical storage size. It is worth mentioning that given a target reliability metric,  the storage size required by the optimal policy with cubic stage cost is about half of what is required by the myopic policy.

\begin{figure}[h]
\centering
  \includegraphics[width=2.5in]{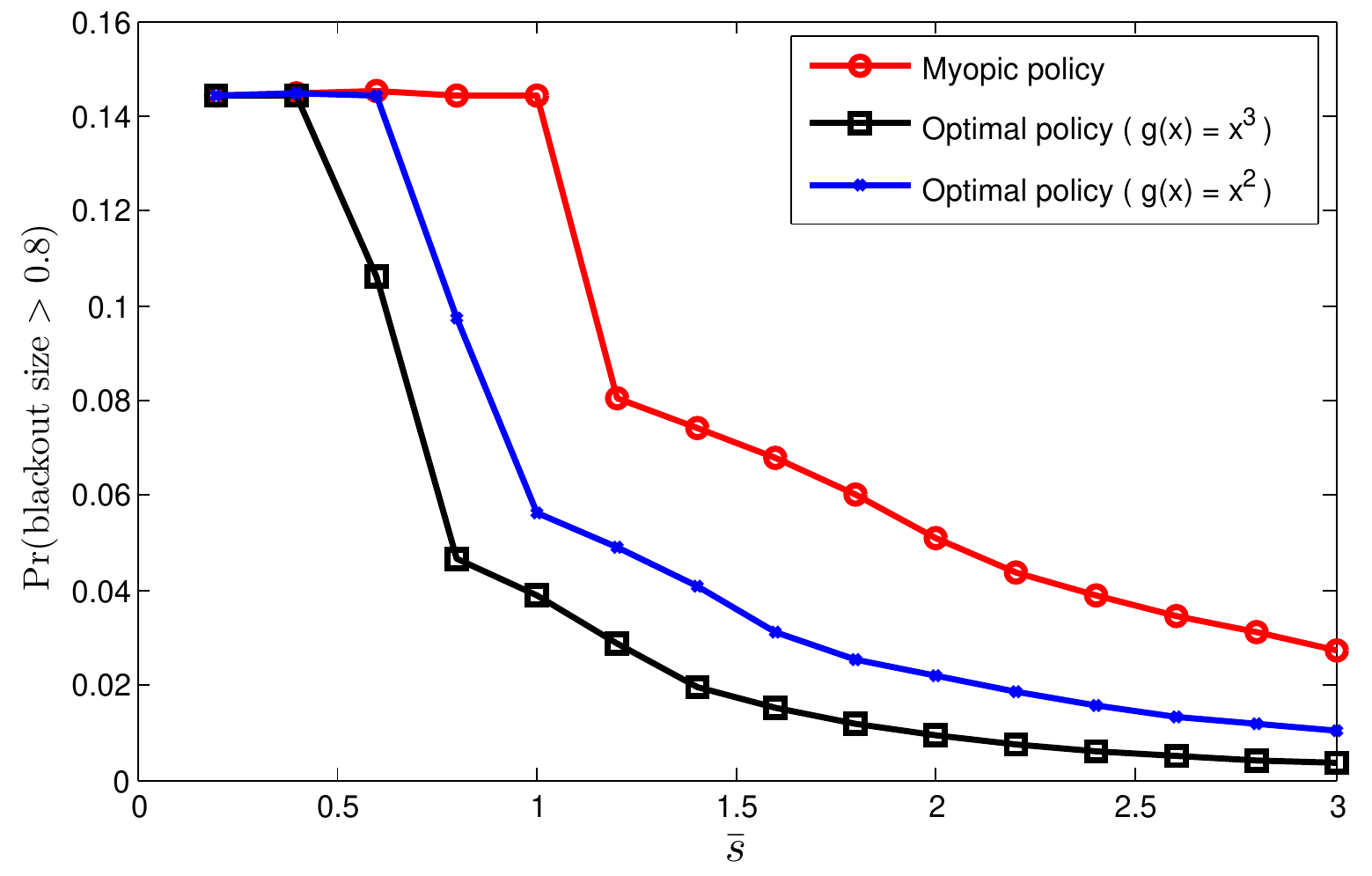}
  \caption{Probability of large blackouts as a function of storage size for different policies (deterministic jumps with rate $Q = 1.0$). }\label{fig:BO_tail_s_bar}
\end{figure}

Finally, we compare the reliability of myopic and non-myopic policies in terms of probability of large blackouts as a function of the volatility of the demand/supply process. We define volatility as the energy of the shock process, i.e.,
$$\textrm{volatility} = Q \E[W^2],$$
which depends both on the mean and variance of the compound arrival process. Figure \ref{fig:BO_tail_volatility} demonstrates large blackout probabilities as a function of volatility, for a system with uniformly distributed jumps with constant mean $R\E[W] = 1$. As shown in Figure \ref{fig:BO_tail_volatility}, higher volatility increases the probability of large blackouts in an almost linear fashion.

\begin{figure}[h]
\centering
  \includegraphics[width=2.5in]{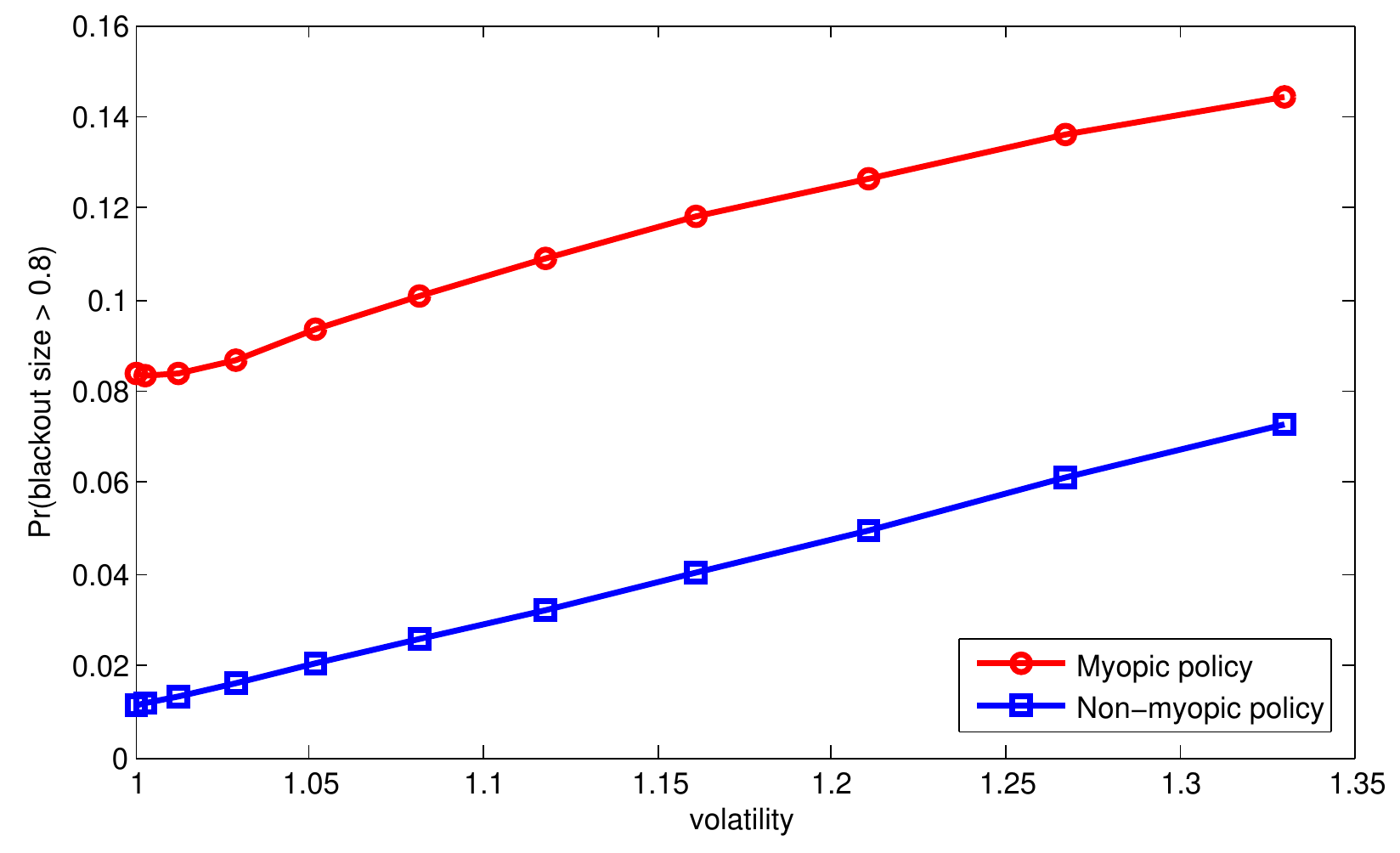}
  \caption{Probability of large blackouts vs. volatility for different policies (uniformly distributed random jumps with $Q = 1.0$ and $E[W] = 1$).}\label{fig:BO_tail_volatility}
  \vspace{-.2in}
\end{figure}

\section{conclusions}\label{sec:conclusions}

We examined the reliability value of storage in a power supply network with uncertainty in supply/demand and upward ramp constraints on both supply and storage. The uncertainty was modeled as a compound poisson arrival of energy deficit shocks. We formulated the problem of optimal control of storage for maximizing system reliability as minimization over all stationary Markovian control policies, of the infinite horizon expected discounted cost of blackouts. We showed that for a linear stage cost, a \emph{myopic} policy which uses storage to compensate for all shocks regardless of their size is optimal. However, for strictly convex stage costs the myopic policy is not optimal. Our results suggest that for high ratios of the average rate of shock size to storage ramp rate, there is a critical level of storage size above which, the value of additional capacity quickly diminishes. For ratios around three and above, there seems to be another critical level below which, storage capacity provides very little value. Finally, Our results suggest that for all control policies, there seems to be a critical level of storage size, above which the probability of suffering large blackouts diminishes quickly.

\appendix

{\bf Proof of Theorem \ref{thm:HJB}}:
\emph{Part (i):} The monotonicity property of the value function follows almost immediately from the definition. Let $0 \leq s_1 < s_2 \leq \bar s$, and assume  $C(s) = C_{\mu}(s)$ for some policy $\mu$. Given the initial state $s_1$, let $u^{(1)}_t$ be the control process under policy $\mu$. Note that for every realization $\omega$ of the compound Poisson process, the sample path $u^{(1)}_t(\omega)$ is admissible for initial condition $s_2 > s_1$. Therefore, by definitions (\ref{discounted_cost}) and (\ref{value_function}), we have $C(s_2) \leq C(s_1)$.

In order to show the strict monotonicity, consider the controlled process starting from $s_1$. Let $\tau$ be the first arrival time such that $g( W_\tau - u^{(1)}_\tau) > 0$. By Assumption \ref{assump:stage_cost}, we have $\pr (\tau \in [0,T]) > 0$ for some $T < \infty$. For every sample path $\omega$, define the control process
$$u^{(2)}_t(\omega) = u^{(1)}_t(\omega) + \delta \cdot \mathbb I_{\{t = \tau(\omega)\}},$$
for some $\delta > 0$ such that $\delta \leq \min\{s_2-s_1, W_{\tau(\omega)} - u^{(1)}_{\tau(\omega)}\}$.

It is clear that $u^{(2)}_t(\omega)$ is admissible for the controlled process starting from $s_2$. Using the definition of the expected cost function in (\ref{discounted_cost}), we can write
\begin{eqnarray*}
C(s_1) - C(s_2) &=& \E_\omega[e^{-\theta \tau(\omega)}g(W_{\tau(\omega)} - u^{(1)}_{\tau(\omega)}) \\
&& - e^{-\theta \tau(\omega)} g(W_{\tau(\omega)} - u^{(1)}_{\tau(\omega)} - \delta)] \\
&\geq& \E[\e e^{-\theta \tau(\omega)}], \quad \textrm{for some } \e > 0 \\
& \geq& \e e^{-\theta T} \pr (\tau \in [0, T])  > 0,
\end{eqnarray*}
where the first inequality holds by strict monotonicity of $g$.


\emph{Part (ii):}
We first prove convexity of $J^*(s,w)$ defined in Theorem \ref{thm:bellman}, and use it to establish convexity of $C(s)$.

In order to show convexity of $J^*(s,w)$, we need to show that the operator $T$ defined in (\ref{bellman_eq}) preserves convexity. Then the claim would be immediate using the convergence of value iteration algorithm (\ref{value_iteration}) to optimal cost $J^*$, where the initial condition is an arbitrary convex function such as $J_0 = 0$.

Next we show that the operator $T$ preserves convexity for this particular problem. Define the objective function in (\ref{bellman_eq}) as $Q(s,w,u)$. We have
\begin{eqnarray*}
\hspace{-3pt}  Q(s,w,u) \hspace{-9pt} &=& \hspace{-9pt} g(w - u) + \E\Big[e^{-\theta t_0} J\big(\min\{s - u + r t_0, \bar s\}, W\big)  \Big] \\
&=& \hspace{-8pt} g(w - u) + \int_{ \frac{\bar s - s +u}{r}}^\infty e^{-\theta t_0} \E[J\big(\bar s, W\big)] Re^{-Q t_0} dt_0 \\
&& \hspace{-18pt} +\ \int_0^{ \frac{\bar s - s +u}{r}} e^{-\theta t_0} \E[J\big(s - u + r t_0, W\big)] Re^{-Q t_0} dt_0.
\end{eqnarray*}

Using the fact that $J$ is convex, linearity of expectation and basic definition of a convex function, it is straightforward but tedious to show that $Q(s,w,u)$ is a convex function. We omit the details for brevity. Given the convexity of $Q$, the convexity of $(TJ)(s,w)$ is immediate, since we are minimizing a multidimensional convex function over one of its dimensions. Hence, we have established convexity of $J^*(s,w)$ in $(s,w)$. Finally, we can express $C(s)$ in terms of $J^*(s,w)$ as in (\ref{J_C_relation}). This results in convexity of $C(s)$ using the above argument for proving convexity of $Q(s,w,u)$.


\emph{Part (iii):} The derivation of Hamilton-Jacobi-Bellman is relatively standard. We omit the proof for brevity, and present a proof sketch based on \emph{principle of optimality} in \cite{ACC_report}. For a more detailed treatment, please refer to \cite{DP_book}, \cite{Fleming} and \cite{Elliott}.

{\bf Proof of Theorem \ref{thm:linear_policy}}:
We establish optimality of $\mu^*$ by showing that it achieves an expected cost no higher than any other admissible policy.
Consider an admissible policy $\tilde \mu$ such that $\tilde\mu(s,w) < \min\{s,w\}$ for some $(s,w) \in [0,\bar s] \times [0,B]$. For every sample path of the controlled process, let $\tau_1(\omega)$ be the first Poisson arrival time such that
$$\min\{s_{\tau_1^-}, W_{\tau_1}\} - \tilde \mu(s_{\tau_1^-}, W_{\tau_1}) = \e > 0.$$

Therefore, by applying policy $\tilde \mu$ instead of $\mu^*$, we pay an extra penalty of $\beta \e e^{-\theta \tau_1(\omega)}$. The reward for this extra penalty is that the state process is now biased by at most $\e$, which allows us to avoid later penalties. However, since the stage cost is linear, the penalty reduction by this bias for any time $\tau_2(\omega) > \tau_1(\omega)$ is at most $\beta \e e^{-\theta \tau_2(\omega)}$. Hence, for this sample path $\omega$, the policy $\tilde \mu$ does worse than the myopic policy $\mu^*$ at least by
$\beta \e (e^{-\theta \tau_1(\omega)} -  e^{-\theta \tau_2(\omega)}) > 0.$
Therefore, by taking the expectation for all sample paths, the myopic policy cannot do worse than any other admissible policy. Note that this argument does not prove the uniqueness of $\mu^*$ as the optimal policy. In fact, we may construct optimal policies that are different from $\mu^*$ on a set $A \subseteq [0,\bar s] \times [0,B]$, where $\pr( (s_{t^-}, W_t) \in A ) = 0$.
\bqed

We delay the proof of Theorem \ref{thm:policy_monotonicity} until after proof of Theorem \ref{thm:collapse}. Let us start with some useful lemmas on the structure of the kernel function.

\begin{lemma}\label{lemma:phi_sticky}
Let $\phi(p)$ be defined as in (\ref{phi_def}). We have
\begin{enumerate}
  \item If $\phi(p_0) = - p_0$ for some $p_0$, then
      $$\phi(p) = -p, \quad \foral p \leq p_0.$$
  \item If $\phi(p_1) = 0$ for some $p_1$, then
      $$\phi(p) = 0, \quad \foral p \geq p_1.$$
\end{enumerate}
\end{lemma}
\begin{proof}
By convexity of the stage cost function and Theorem \ref{thm:HJB}(ii),  $\phi(p)$ is the optimal solution of a convex program. Therefore, if $\phi(p_0) = -p_0$ for some $p_0 \leq 0$, we have
$$g'(-p_0) + C'(0) \geq 0.$$
Thus, by convexity of stage cost, $g(-p) \geq g(-p_0)$, for any $p \leq p_0$. Therefore, by convexity of $C(\cdot)$ and $g(\cdot)$,
$$g'(x) + C'(x + p) \geq g'(-p) + C'(0) \geq 0, \foral x \geq -p,$$
which immediately implies optimality of $(-p)$, for $p \leq p_0$.

Similarly, for the case where $\phi(p_1) = 0$, we have
$g'(0) + C'(p_1) \geq 0,$
which implies
$$g'(x) + C'(x + p) \geq g'(0) + C'(p) \geq 0, \quad \foral p \geq p_1,$$
hence, the objective is nondecreasing for all feasible $x$ and $\phi(p) = 0$.
\end{proof}

\begin{lemma}\label{lemma:dC_bound}
Let $C(s)$ be defined as in (\ref{value_function}), and assume that the stage cost $g(\cdot)$ is convex. Then
\begin{equation}\label{dC_ineq}
    \frac{dC}{ds}(s) \geq -\frac{Q}{r} \E_W[g(W)], \quad 0 \leq s \leq \bar s.
\end{equation}
\end{lemma}
\begin{proof}
By Theorem \ref{thm:HJB}(ii), the optimal cost function $C(s)$ is convex. Hence,
$\frac{dC}{ds}(s) \geq \frac{dC}{ds}(0).$
On the other hand, by Theorem \ref{thm:HJB}(iii), we can write
\begin{eqnarray*}
 \frac{dC}{ds}(0) &=& \frac{Q+\theta}{r}  C(0) -  \frac{Q}{r} \E_W\Big[\min_{u = 0} g(W-0) + C(0) \Big].
\end{eqnarray*}
Combining the two preceding relations proves the claim.
\end{proof}

\begin{lemma}\label{lemma:inactive_const}
If Assumption \ref{assump:neg_drift} holds, then the first constraint in (\ref{phi_def}) is never active, i.e., $\phi(p) < \min\{B, \bar s - p\}$.
\end{lemma}
\begin{proof}
We show that under Assumption \ref{assump:neg_drift}, the slope of the objective function is always non-negative at $x = \min\{B, \bar s - p\}$. In the case where $\bar s - p \leq B$, we have
$$\frac{\partial}{\partial x}\Big(g(x) + C(x+p)\Big)\Big|_{x = \bar s - p} = g'(\bar s - p) + C'(\bar s) \geq 0,$$
where the inequality follows from monotonicity of $g$ and (\ref{HJB_BC}). For the case where $\bar s - p \geq B$, we employ Lemma \ref{lemma:dC_bound} and Assumption \ref{assump:neg_drift} to write
\begin{eqnarray*}
&&\hspace{-15pt} \frac{\partial}{\partial x}\Big(g(x) + C(x+p)\Big)\Big|_{x = B} = g'(B) + C'(B+p) \\
&&\geq g'(B)  -\frac{Q}{r} \E_W[g(W)] \geq g'(B)  -\frac{\E_W[g(W)]}{\E[W]} \geq 0,
\end{eqnarray*}
where the last inequality holds because $g(w) \leq w g'(B)$, for all $w \leq B$, which is a convexity result.
\end{proof}

\vspace{.1in}
{\bf Proof of Theorem \ref{thm:collapse}}:
By Theorem \ref{thm:HJB}(iii), we can characterize the optimal policy as
\begin{eqnarray}\label{policy_convex_program}
   \mu^*(s,w) &=&  \textrm{argmin} \ \   g(w-u) + C(s-u)  \\
  && \textrm{      s.t.    }  0\leq u \leq \min\{s,w\}. \nonumber
\end{eqnarray}

Note that the optimization problem in (\ref{policy_convex_program}) is convex, because $g(\cdot)$ and hence, $C(\cdot)$ is convex (cf. Theorem \ref{thm:HJB}(ii)). Using the change of variables
$$x = w - u, \quad p = s - w,$$
we can rewrite (\ref{policy_convex_program}) as $\mu^*(s,w) = w - x^*(p,w)$, where
\begin{eqnarray}\label{transformed_problem}
  x^*(p,w) &=& \textrm{argmin} \ \   g(x) + C(p+x)  \\
  &&  \textrm{      s.t.    } x \geq \max\{0, -p\} \nonumber \\
  &&  \qquad x \leq w. \nonumber \label{extra_const}
\end{eqnarray}

The optimization problem in (\ref{transformed_problem}) depends on both parameters $p$ and $w$. We may remove the dependency on $w$ as follows. Since $w \leq B, \bar s - p$, we may relax the last constraint, $x \leq w$, by replacing it with
$x \leq \min\{B, \bar s - p\} $
The optimal solution of the relaxed problem is the same as $\phi(p)$ defined in (\ref{phi_def}). If $\phi(p) < w$, then the relaxed constraint is not active, and $\phi(p)$ is also the solution of (\ref{transformed_problem}). Otherwise, since we have a convex problem, the constraint $x \leq w$ must be active, which uniquely identifies the optimal solution as $w$. Therefore, the optimal solution of the problem in (\ref{transformed_problem}) is given by $x^*(p,w) = \min\{\phi(p), w\}$. Combining the preceding relations, we obtain
$$\mu^*(s,w) = w -  \min\{\phi(s-w), w\} = \Big[w - \phi(s-w)\Big]^+.$$

The representation in (\ref{kernel_partition}) is  a direct consequence of Lemmas \ref{lemma:phi_sticky} and \ref{lemma:inactive_const}. Between some break-points $b_0$ and $b_1$, the optimal solution of (\ref{phi_def}) can only be an interior solution, which is given by (\ref{interior_eq}). The uniqueness of  $\phi^\circ(p)$ follows from strict convexity of $g$. Finally, by continuous differentiability of the cost function, equation (\ref{interior_eq})  should hold at the break-points as well. Therefore,
$$  g'(b_0)  + C'(b_0+ (-b_0)) = 0,\quad  g'(0)  + C'(0 + b_1) = 0,$$
which is equivalent to the characterizations in (\ref{b0}) and (\ref{b1}). The first inequality in (\ref{b0}) holds by Lemma \ref{dC_ineq} and convexity of $g(\cdot)$, and the second inequality holds by Assumption \ref{assump:neg_drift} and applying convexity of $g(\cdot)$ again.
\bqed
\begin{lemma}\label{lemma:phi_monotonicity}
Let $\phi(p)$ be defined as in (\ref{phi_def}), and assume that Assumption \ref{assump:neg_drift} holds and the stage cost $g(\cdot)$ is strictly convex. Then for all $p_1 \leq p_2$,
\begin{equation}\label{phi_monotonicity}
    -(p_2 - p_1) \leq \phi(p_2) - \phi(p_1) \leq 0.
\end{equation}
\end{lemma}
\begin{proof}
We first establish the monotonicity of $\phi(p)$. Let $p_1 < p_2$. Given the structure of the kernel function in (\ref{kernel_partition}), there are multiple cases to consider, for most of which the claim is immediate using (\ref{kernel_partition}). We only present the case where $-B \leq p_1 \leq b_1$ and $b_0 \leq p_2 \leq b_1$. A necessary optimality condition at $p_1$ is given by
\begin{equation}\label{p1_nec_optimality}
    g'(\phi(p_1)) + C'(p_1 + \phi(p_1)) \geq 0.
\end{equation}

Similarly, for $p_2$, we must have
\begin{equation}\label{p2_nec_optimality}
    g'(\phi(p_2)) + C'(p_2 + \phi(p_2)) = 0,
\end{equation}

Now, assume $\phi(p_2) > \phi(p_1)$. By convexity of $C(\cdot)$ (cf. Theorem \ref{thm:HJB}(ii)) and strict convexity of $g(\cdot)$,  we obtain
$$g'(\phi(p_2)) + C'(p_2 + \phi(p_2)) > g'(\phi(p_1)) + C'(p_1 + \phi(p_1)) \geq 0,$$
which is a contradiction to (\ref{p2_nec_optimality}).

For the second part of the claim, again, we should consider several cases depending on the interval to which $p_1$ and $p_2$ belong. Here, we present the case where $b_0 \leq p_1 \leq b_2$ and $b_0 \leq p_2 \leq \bar s$. The remaining cases are straightforward using (\ref{kernel_partition}). In this case, we have
\begin{equation}\label{p1_nec_optimality2}
    g'(\phi(p_1)) + C'(p_1 + \phi(p_1)) = 0,
\end{equation}
\vspace{-.2in}
\begin{equation}\label{p2_nec_optimality2}
    g'(\phi(p_2)) + C'(p_2 + \phi(p_2)) \geq 0.
\end{equation}
Combine the optimality conditions in (\ref{p1_nec_optimality2}) and (\ref{p2_nec_optimality2}) to get
\begin{equation}\label{combined_optimality}
        g'(\phi(p_2)) + C'(p_2 + \phi(p_2)) \geq     g'(\phi(p_1)) + C'(p_1 + \phi(p_1))
\end{equation}

Assume $\phi(p_2) < \phi(p_1)$; otherwise, the claim is trivial. By strict convexity of $g(\cdot)$, we have $g'(\phi(p_2)) < g'(\phi(p_1))$. Therefore by (\ref{combined_optimality}), it is true that
\begin{equation}\label{strict_Cp_relation}
    C'(p_2 + \phi(p_2)) > C'(p_1 + \phi(p_1)).
\end{equation}
Now assume $\phi(p_2) - \phi(p_1) < -(p_2 - p_1)$. By rearranging the terms of this inequality and invoking the convexity of $C(\cdot)$, we get
$  C'(p_2 + \phi(p_2)) \leq C'(p_1 + \phi(p_1)),$
which is in contradiction to (\ref{strict_Cp_relation}). Therefore, the claim holds.
\end{proof}
\vspace{.1in}
{\bf Proof of Theorem \ref{thm:policy_monotonicity}}:
First, note that by Lemma \ref{lemma:phi_monotonicity}, we get
$$\phi(s_2 - w) \leq \phi(s_1 - w), \quad \foral w, s_1 \leq s_2$$
which implies (cf. Theorem \ref{thm:collapse})
$$\mu^*(s_2,w)\hspace{-3pt} = \big[w - \phi(s_2 - w)]^+ \geq \big[w - \phi(s_1 - w)]^+\hspace{-3pt} = \mu^*(s_1,w).$$

Moreover, for all $s$ and $w_1 \leq w_2$, we can use the second part of Lemma \ref{lemma:phi_monotonicity} to conclude
$$\phi(s - w_1) - \phi(s - w_2) \geq -(w_2 - w_1).$$
By rearranging the terms, it follows that
$$\mu^*(s,w_2)\hspace{-3pt} =\hspace{-3pt} \big[w_2 - \phi(s - w_2)]^+ \hspace{-3pt}\geq\hspace{-3pt} \big[w_1 - \phi(s - w_1)]^+ \hspace{-3pt} = \mu^*(s,w_1),$$
which completes the proof. \bqed

\bibliographystyle{unsrt}
\bibliography{reliabilityref}

\end{document}